\documentclass[12pt]{article}
\usepackage{amsmath}
\usepackage{graphicx,psfrag,epsf}
\usepackage{enumerate}
\usepackage{natbib}

\usepackage[utf8]{inputenc}
\usepackage{amsthm}
\usepackage{mathtools}
\usepackage{multirow}
\usepackage{caption}
\usepackage{subcaption}
\usepackage{color}
\usepackage{float} 

\newtheorem*{mypropo*}{Proposition}

\newtheorem*{mycoro*}{Corollary}

\newcommand{\blind}{0}

\begin{document}

\def\spacingset#1{\renewcommand{\baselinestretch}%
{#1}\small\normalsize} \spacingset{1}



\if0\blind
{
  \title{\bf The Sequential Normal Scores Transformation}
  \author{Dr. William J. Conover
    \\
    Department of Mathematics and Statistics, Texas Tech University\\
    \\
    Dr. V\'ictor G. Tercero-G\'omez\\
    School of Engineering and Sciences, Tecnologico de Monterrey\\
    \\
    Dr. Alvaro E. Cordero-Franco \\
    Facultad de Ciencias F\'isico Matem\'aticas, Universidad Autonoma de Nuevo Leon\\}
  \maketitle
} \fi

\if1\blind
{
  \bigskip
  \bigskip
  \bigskip
  \begin{center}
    {\LARGE\bf The Sequential Normal Scores Transformation}
\end{center}
  \medskip
} \fi

\bigskip
\begin{abstract}
The sequential analysis of series often requires nonparametric procedures, where the most powerful ones frequently use rank transformations. Re-ranking the data sequence after each new observation can become too intensive computationally. This led to the idea of sequential ranks, where only the most recent observation is ranked. 
However, difficulties finding, or approximating, the null distribution of the statistics may have contributed to the lack of popularity of these methods.
In this paper, we propose transforming the sequential ranks into sequential normal scores which are independent, and asymptotically standard normal random variables. Thus original methods based on the normality assumption may be used.

A novel approach permits the inclusion of a priori information in the form of quantiles. It is developed as a strategy to increase the sensitivity of the scoring statistic.
The result is a powerful convenient method to analyze non-normal data sequences. Also, four variations of sequential normal scores are presented using examples from the literature.
Researchers and practitioners might find this approach useful to develop nonparametric procedures to address new problems extending the use of parametric procedures when distributional assumptions are not met. These methods are especially useful with large data streams where efficient computational methods are required.
\end{abstract}

\noindent%

\noindent%
\vfill

\newpage
\spacingset{2} 

\section{Introduction}
\label{sec:intro}
Many applications of statistics involve sequences of observations, where
decisions are made at
several time points along the sequence.
According to \cite{wald1945sequential}, a sequential test of hypothesis is
the name given to any procedure, where at the $n$-th trial of an experiment
a decision is made to accept the null hypothesis, accept the alternative
hypothesis, or continue the experiment by making another observation.
Hence, the analysis is sequential, and the time it takes to reach a
statistical conclusion is random in nature.
\cite{dodge1929method}
developed sequential tests with applications in quality inspection.
\cite{shewhart1931} suggested the idea of online monitoring
and continuous testing for statistical control. \cite{wald1945sequential}
conceived a general testing procedure named the sequential probability ratio test.
\cite{box1957evolutionary} proposed an evolutionary operation, a sequential
method to carry on experiments. From these initial propositions, many
statistical approaches have been created to solve a variety of problems.
However, most of these statistical methods,
when employed in practice, assume knowledge of the probability distribution
of the observations, or that the observations can be averaged across groups
so the sample mean is approximately normally distributed by the Central Limit
Theorem (refer to Shewhart $\bar{X}$-charts as an example).

When the data are not normally distributed, the normal-based methods are often
robust to the non-normality by the Central Limit Theorem, but robustness
protects only the validity of the probability of a Type I error, and such
methods can suffer from a loss of power. Nonparametric methods are extensively
used in analysis of fixed-sample-size
applications, with the well-known advantage of often increasing the power over
parametric methods, especially in the presence of outliers or skewed data. The most popular and most powerful nonparametric methods
usually involve replacing the observations by ranks, or scores based on ranks.
Nevertheless, a major problem arises when trying to adapt nonparametric
rank methods
to sequential data due to the intensive computation required if all the data
are reranked each time a new nonparametric test is performed. Some solutions
involve grouping the observations so the sequence is not observed
continuously, but only at selected intervals \citep{ross_2011}. However, this
lessens the sensitivity of the analysis by effectively reducing the number of
times the sequence is tested.

From \cite{parent1965sequential}, a reasonable solution to the computation
problem involves using sequential ranks so only the most recent observation
is ranked relative to the previous, unchanged, ranks. \citeauthor{parent1965sequential}
notes that there is a one-to-one correspondence between the ordinary ranks and the
sequential ranks, so no information is lost by using sequential ranks instead of
ordinary ranks. In fact there is an advantage gained, in that sequential ranks are
independent of each other while ordinary ranks have a built-in dependence.

\citeauthor{parent1965sequential} adapted sequential ranks to allow
continuous testing of a
sequence of observations by adapting the  sequential probability
ratio test  which has optimum properties as shown by \cite{wald1948optimum}. The adaptation is awkward and unwieldy, which may explain why it has not
become popular in usage. \cite{reynolds1975sequential} solved the problem of
sequential ranks being “bottom heavy” in
that small ranks are more likely to occur than large ranks, which occur only
as the sequence accumulates more observations. He “standardized” the
sequential ranks by dividing them at each stage by (i+1) where i represents
the number of observations up to that point. This effectively “spread” the
ranks over the unit interval (0,1) at each stage in the analysis. However, the
null distribution of the rank statistics was difficult to find, and so he
showed it converged to a Markov process, and used the theory of Markov
processes to obtain critical values for statistical tests. This too is
awkward and unwieldy and has not been widely accepted in practice.

We are suggesting (for the first time, perhaps) to take advantage of the
relationship between ranks and the estimators of a cumulative probability to
replace the ranks by normal
scores in the sense of \cite{van1952order}, where each sequential
rank is substituted with the corresponding quantile of the standard normal distribution.
Thus, the ranks are effectively replaced by numbers that appear to have come from a
standard normal distribution. Because these sequential normal scores are
independent of each other, and
behave like normal random variables, the usual sequential methods based on
normal observations may be applied as approximate methods,
and special analytical methods are no
longer needed. Changes in location or variation occurring in the original sequence
will be reflected by the changes in shift or spread of the sequential normal
scores.

This paper examines the theoretical validity of sequential normal scores
and evaluates the versatility through the analysis of
known real life examples found in the literature on applications of statistical
methods to the analysis of sequences of observations.
Section \ref{sec:SNS} is meant for a wide audience of practitioners and academics
alike. It describes and discusses four variations of the proposed method with their
corresponding numerical examples that illustrate the flexibility in practical
situations. Descriptions are straight forward, only a basic level of mathematics and
statistical methods is required, and the discussion is centered around the application.

Section \ref{sec:theory} is intended for a specialized audience.
It demands a deeper knowledge of probability and statistical
theory by offering a discussion on the mathematical development of the method and some
of its properties to establish confidence among users that the methods have a sound theoretical
justification and foster further applications.
Readers who are interested only in the application might want to skip this section.
Finally, Section \ref{sec:conclusions} provides a summary of the approach, its
practical implications, and suggestions for future research.

\section{Sequential Normal Scores}
\label{sec:SNS}
Common industrial applications of sequential tests include the statistical
monitoring of individual or batched observations where some knowledge might,
or might not,
be available about the population quantiles. If available, known or estimated quantiles
can be incorporated to increase the sensitivity of a statistic. Otherwise,
a self-starting approach is required to build knowledge as it becomes available.
Both situations are addressed in this section through four related models. The
first two models address the self-starting situation, while the last two aim
to incorporate existing knowledge about a quantile. An example is given immediately
after a model is presented to illustrate and discuss applicability.

\subsection{Individual observations with unknown quantiles}
\label{subsec:model1}

Let $X_1,X_2,...,$ be a sequence of independent identically distributed random
variables with a continuous distribution function $F(x)$. The sequential rank
$R_i$ of $X_i$, where $i$ stands for the observed order within the sequence, is
defined by \cite{parent1965sequential} as the rank of $X_i$ relative to the previous
random variables in the sequence up to and including $X_i$. The sequential ranks of all the observations preceding $X_i$ remain unchanged. This can save considerable
computing time when using rank-based nonparametric methods.

For reasons explained in Section \ref{sec:theory} we will use 
\begin{align}
\label{eq:Pn}
P_i=\frac{R_i-0.5}{i}
\end{align}
to estimate $F(X_i)$. \textit{Sequential normal scores} are obtained from $P_i$ using
$Z_i = \Phi^{-1}(P_i)$ where $\Phi^{-1}$ stands for the inverse cumulative standard
normal distribution function. As will be proven in Section \ref{sec:theory}, the
sequence $\left\{ Z_i:i=1,\ldots \right\}$  consists of mutually independent
asymptotically (as $i$ gets large) standard normal random variables.

As a short illustration of how sequential normal scores are obtained, consider the
following observations on the first 10 random variables $X_i$ $(i = 1, 2,\ldots, 10)$
in a sequence from an arbitrary continuous distribution function
$F(x)$:4.6, 5.1, 3.9, 4.4, 4.8, 6.6, 5.3, 8.3, 4.7, and 5.0.
The sequential ranks $R_i$ are: 1, 2, 1, 2, 4, 6, 6, 8, 4, and 6.
Note that $F(x)$ is assumed to be continuous so the probability of ties is zero,
but if ties occur due to round off, average ranks can be used in practice.
Hence, the estimates of $F(X_i)$, or $P_i$, are:
0.5000, 0.7500, 0.1667, 0.3750, 0.7000, 0.9167, 0.7857, 0.9375, 0.3889, and 0.5500.
The sequential normal scores $Z_i$ are then (rounded to four decimals):
0.0000, 0.6745, -0.9674, -0.3186, 0.5244, 1.3830, 0.7916, 1.5341, -0.2822, 0.1257.
We will show in Section 3 that if the $X_i$’s in the sequence are independent,
the sequential ranks are independent \citep{parent1965sequential} and therefore the
sequential normal scores are independent and approach in distribution the standard
normal distribution as $i$ gets large.

\subsubsection{Bearing example}
When bearings slide over a lubricated surface, vibrations are steady and
predictable, in control. However, after a large period of utilization, microscopic
fractures initiate a vicious cycle that increases vibrations exponentially until
bearings fatigue and break. Quick detection of a sustained change in bearing 
vibrations
creates an advantage that might allow preventive maintenance to avoid costly
machine repairs. To illustrate the application of sequential normal scores on real measurements, the
procedure was applied over a data set from the IEEE PHM 2012 Data Challenge
organized by the IEEE Reliability Society and FEMTO-ST Institute.
During the challenge to estimate the remaining useful life of bearings,
several experiments to accelerate degradation were carried out on a laboratory
experimental platform called PRONOSTIA. Datasets and further information about
the data challenge are available in \cite{nectoux2012pronostia}. Vibrations
were processed using a fast Fourier transform (FFT) and the results were summarized
as RMS (root mean square) measurements, a convention used to
to provide an indication of the amount of energy spent on vibration. FFT and RMS
measurements were calculated by \cite{barraza2015adaptive}, where time series
models are used to characterize PRONOSTIA's bearing data. The specific data set
used to create this example was taken from \cite{nectoux2012pronostia},
scenario 1, bearing 4. Here, the bearing was run to failure; hence, it is known,
by design, that measurements describe a movement from an in-control state to an
out-of-control state.

Using a least squares approach over the first 1000 RMS observations,
an MA(1) was fitted over the first differences,
and independent errors were used for process monitoring. A CUSUM chart
was constructed after transforming the errors into sequential normal scores. Even if data
is not normal, practitioners can rely on the normal approximation of the scores and proceed
with a CUSUM chart set up for normally distributed observations. The monitored errors
can be seen in Figure \ref{fig:bearingsub1}, and the corresponding CUSUM chart
applied over the sequential normal scores is shown in Figure \ref{fig:bearingsub2}.
A one-sided CUSUM for positive shifts with an allowance $k = 0.25$
and a decision interval $H = 7.267$ were used for monitoring to achieve an
approximate in control $ARL$ of 500. A signal is triggered at observation
1082, just before vibrations start to exhibit an evident ``violent'' behavior.

\begin{figure}
\centering
\begin{subfigure}{.5\textwidth}
  \centering
  \includegraphics[width=1\linewidth]{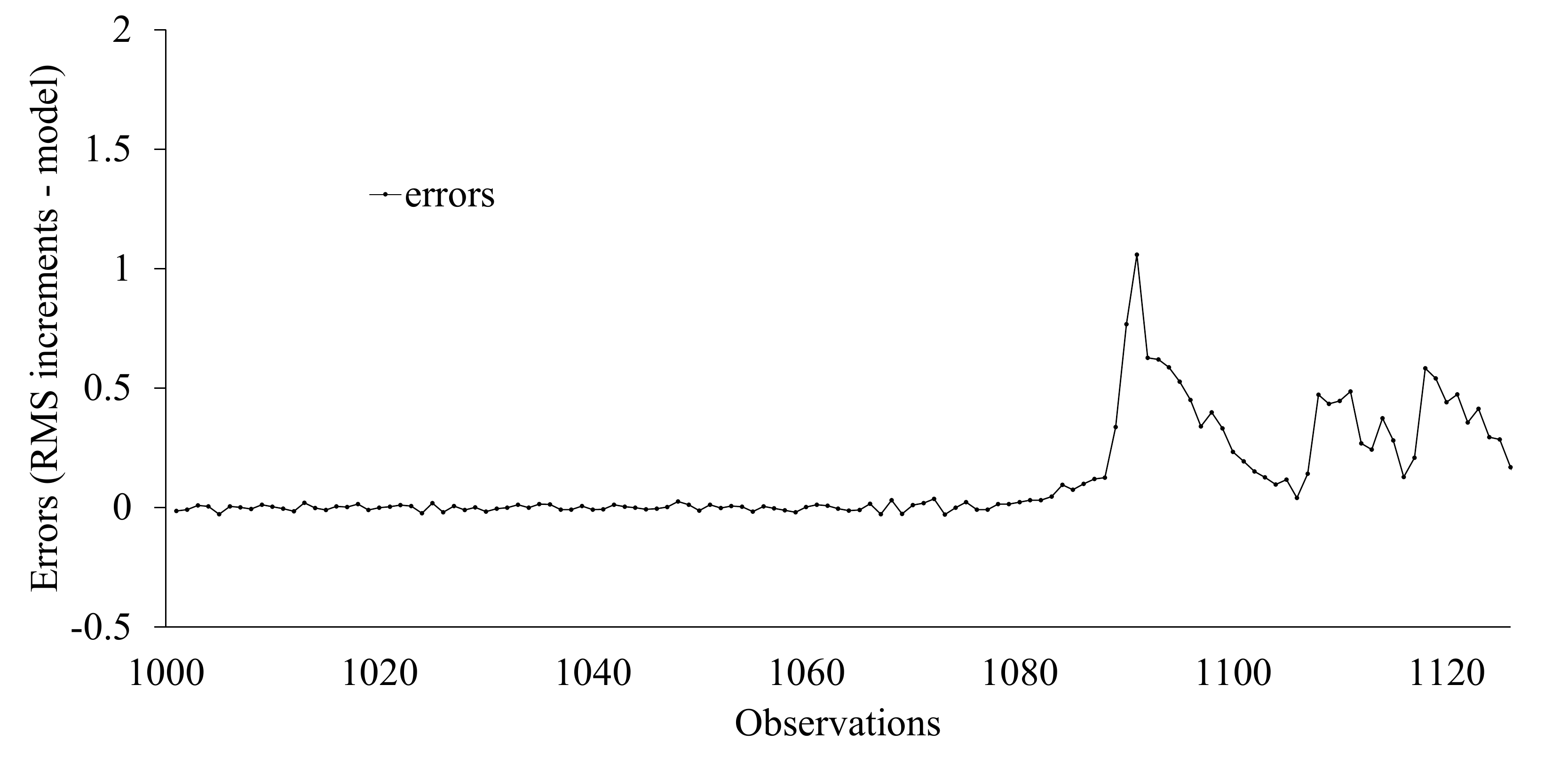}
  \caption{Monitored errors from bearing vibrations.}
  \label{fig:bearingsub1}
\end{subfigure}%
\begin{subfigure}{.5\textwidth}
  \centering
  \includegraphics[width=1\linewidth]{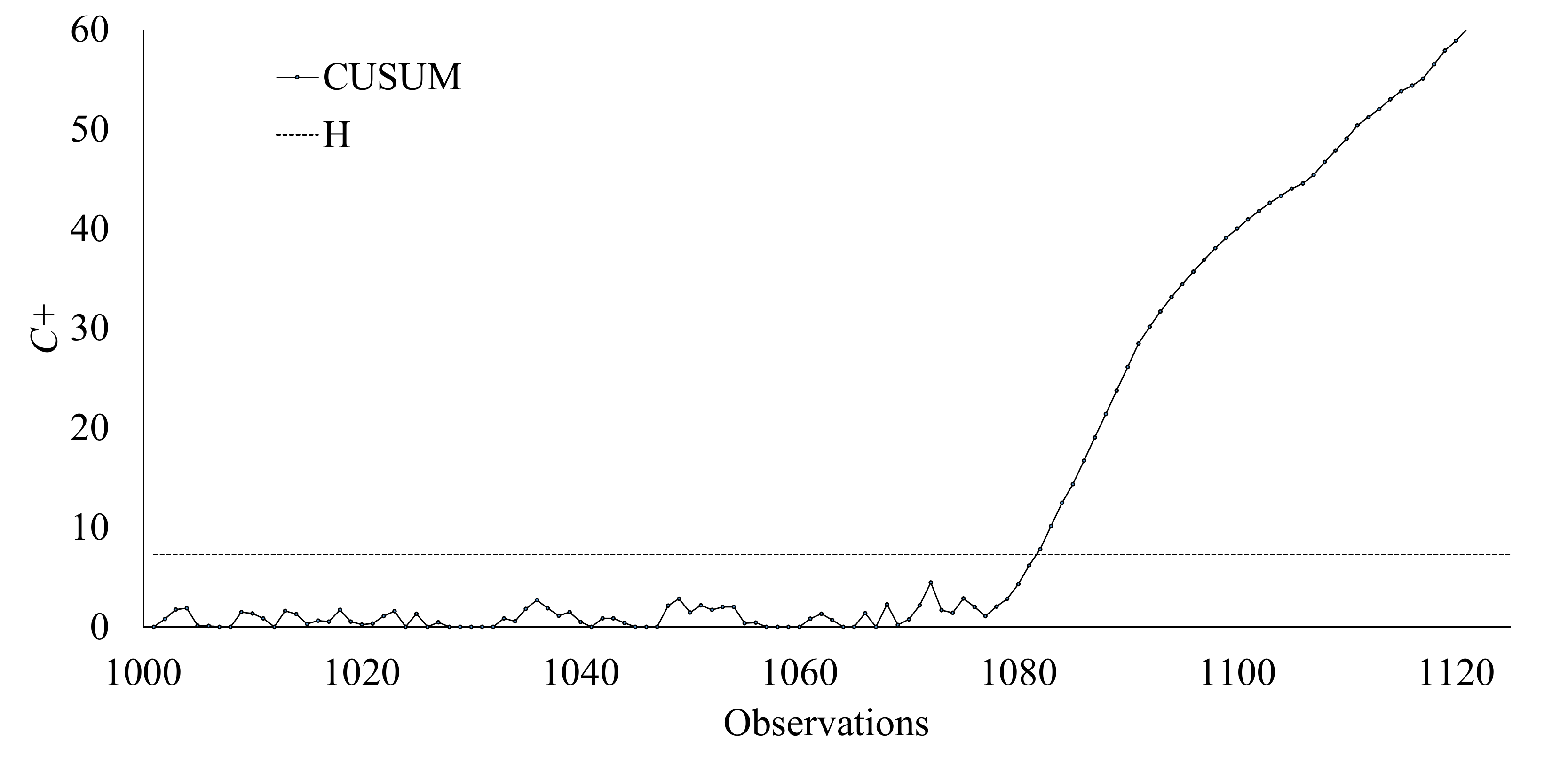}
  \caption{One-sided CUSUM using sequential normal scores.}
  \label{fig:bearingsub2}
\end{subfigure}
\caption{A CUSUM chart was used on sequential normal scores obtained from observed errors
after fitting an ARIMA(0,1,1) model over a first set of 1000 in-control RMS vibration
measurements from bearings at the IEEE PHM 2012 Data Challenge.
A one-sided CUSUM with a reference value $k = 0.25$ and a decision interval $H = 7.267$
was used.}
\label{fig:bearing}
\end{figure}

\subsection{Batched observations with unknown quantiles}
\label{subsec:model2}
Let $\{ X_{ij}:i=1,\hdots,; j=1,\hdots,m \}$ be a sequence of independent identically
distributed random variables with continuous distribution function $F(x)$. The second
subscript refers to the fact that these random variables are grouped into batches
(samples) of size $m$. For the first batch ($i = 1$) the $m$ random variables are 
ranked relative to the other random variables in that batch, and the ranks are
denoted by $R_{1,j}$ for $j=1,\ldots, m$. For all subsequent batches the sequential
rank $R_{i,j}$ of $X_{i,j}$, where $i$ stands for the batch number, and $i\geq 2$,
is given by, 
\begin{align}
\label{eq:Rnj}
R_{i,j}(X_{i,j})= \sum_{k=1}^{i-1} \sum_{l=1}^m I(X_{k,l} \leq X_{i,j}) + 1
\end{align}
for $i\geq 2$ and $j=1,2,\ldots,m$, where $I(X_{kl} \leq X_{i,j})$ is an indicator function.

This sequential rank is computed for each $j=1,\ldots, m$ in batch $i$.
Note that several of these sequential ranks within a batch may be equal to one another,
but do not change as more batches are observed. Again, this can save considerable
computing time when using rank-based nonparametric methods.

For reasons explained in Section \ref{sec:theory} we will use
\begin{align}
\label{eq:P1j}
P_{1,j}=\frac{\left(R_{1,j}-0.5\right)}{m},
\end{align}
and
\begin{align}
\label{eq:Pij}
P_{i,j}=\frac{\left(R_{i,j}-0.5\right)}{m(i-1) +1}
\end{align}  
for $i>1$, to estimate $F(X_{i,j})$. \textit{Sequential normal scores} are obtained from
$P_{i,j}$ using $Z_{i,j}=\Phi^{-1} \left( P_{i,j} \right)$ where $\Phi^{-1}$
stands for the inverse cumulative standard normal distribution function. As will be
proven in Section \ref{sec:theory}, the sequence
$\left\{ Z_{i,j}:i=1,2, \ldots; j=1,2,\ldots,m \right\}$ consists of mutually
independent asymptotically (as $i$ gets large) standard normal random variables.
By comparing each random variable with only the previous batches plus itself,
and not with other random variable in the same batch, the independence of the
sequential ranks is kept, and the sensitivity to detect a change is improved
by not using comparisons within a batch potentially from the alternative distribution.

\subsubsection{Service time example}

We use an example from \cite{yang2012new} and \cite{yang2015new}
in which the efficiency of the service system in a bank is analyzed.
The service process times for 10 counters were measured (in minutes) every 2 days for 30 days.
Fifteen in-control samples of size $m=10$ were obtained from a bank branch.
\cite{yang2015new} show that data appear to be right-skewed (see Figure \ref{fig:timesub1}),
hence a procedure
based on the normal distribution is not recommended. Later, 10 new samples from
a new automatic service system were collected. From the analysis carried out by
\cite{yang2015new} there is indeed a change in the new samples. That is, the 10 new
samples that belong to a out-of-control state showed a reduction in the variance.
By performing a Phase I analysis, \cite{yang2015new} estimated the variance of the
process from the first 15 samples, and carried a test on the remaining 10 samples.
The test was rejected after the second sample.

Without using a Phase I analysis to estimate an in-control value of the variance to use as reference,
sequential normal scores can be adapted to test for a
variance change. Scores $Z_{i,j}$ can be used in an optimal CUSUM for downward process
variance by monitoring the statistic
\begin{align*}
C_i^{-}=min(0,C_{i-1}^{-}+s_i^2 - k)
\end{align*}
where $C_0=0$, $s_i^2$ corresponds to the sample variance (from the scores) in batch i, and
	\begin{align*}
k=\frac{2\log(\frac{\sigma_0}{\sigma_1})\sigma_0^2 \sigma_1^2}{\sigma_1^2-\sigma_0^2},
\end{align*}
where $\sigma_0^2$ and $\sigma_1^2$ corresponds to the out-of-control and the in-control variance,
respectively; which gives a signal when $C_i^-<H$. From \cite{hawkins2012cumulative} we obtain
a value of $H=-1.645$ and an allowance $k = 0.793$ to achieve an $ARL_0$ of $370.0$,
which is optimum for a variance reduction from 1 to 0.8. The value, 0.8, for the out-of-control
variance was selected arbitrarily to illustrate the approach.
Original observations and the results from the CUSUM are plotted in
Figure \ref{fig:time}.
It can be seen that a variance reduction was
signaled at observation 22, that is, 7 observations after the actual change.
Note that the natural CUSUM change-point statistic, the greatest observation with
cumulative value $C^-$
equal to zero, is 16, the actual first out-of-control sample.
In this example, the first batch was transformed using
their corresponding relative ranks, and the subsequent batches were ranked using sequential ranks.
Then the inverse of a standard normal cumulative distribution was evaluated on the
resulting $P_{i,j}$ as defined in equations~\eqref{eq:P1j} and \eqref{eq:Pij}.

If a priori information exists about the null distribution, such as the one obtained
from a Phase I analysis as done in \cite{yang2015new}, historical data, or
a target median, an adaptation of the sequential normal scores to incorporate quantile information
can be used to increase the efficiency of a test
procedure. This adaptation is illustrated in the following two subsections.

\begin{figure}
\centering
\begin{subfigure}{.5\textwidth}
  \centering
  \includegraphics[width=1\linewidth]{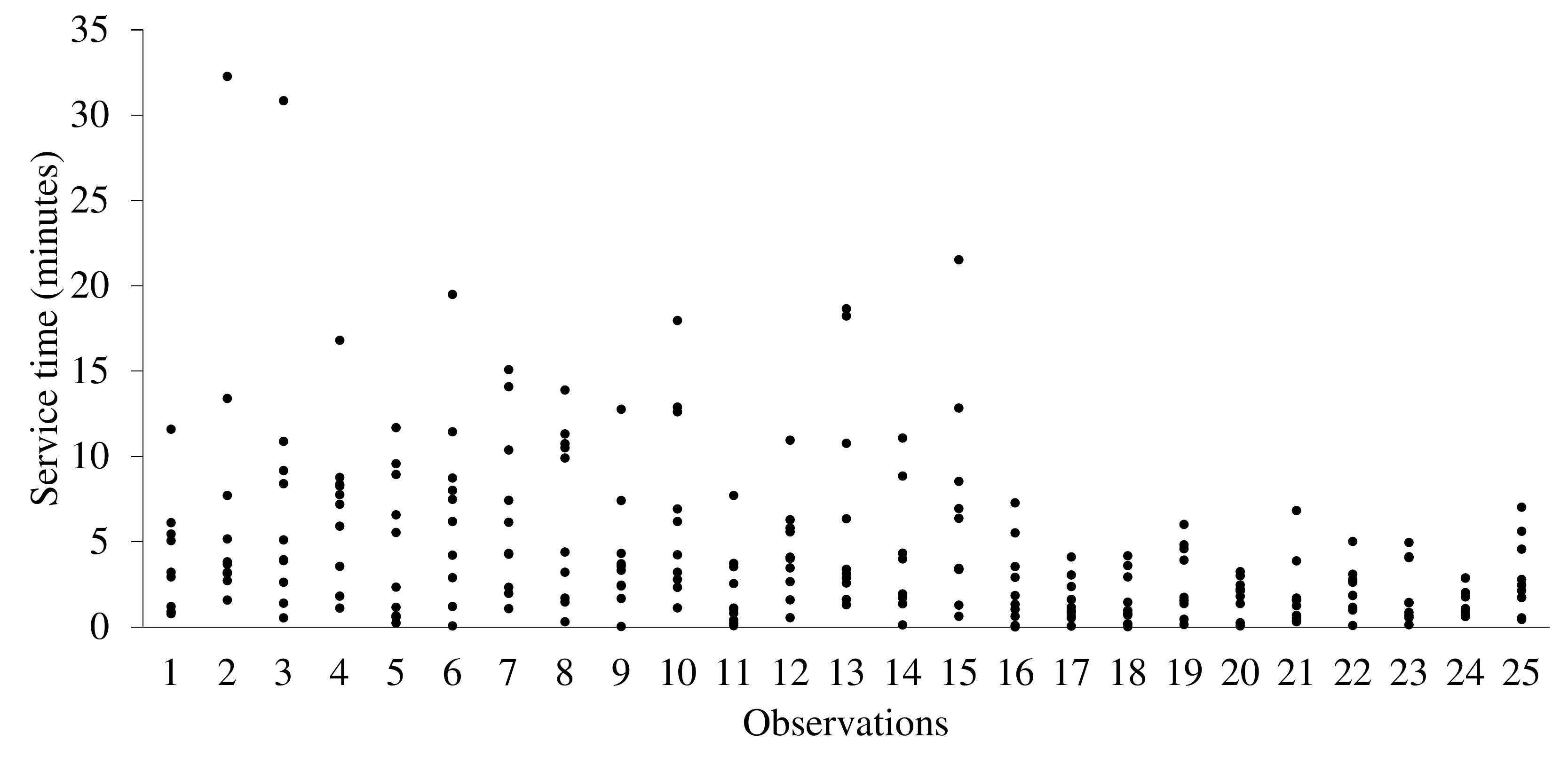}
  \caption{Original measurements from the Bank service times example.}
  \label{fig:timesub1}
\end{subfigure}%
\begin{subfigure}{.5\textwidth}
  \centering
  \includegraphics[width=1\linewidth]{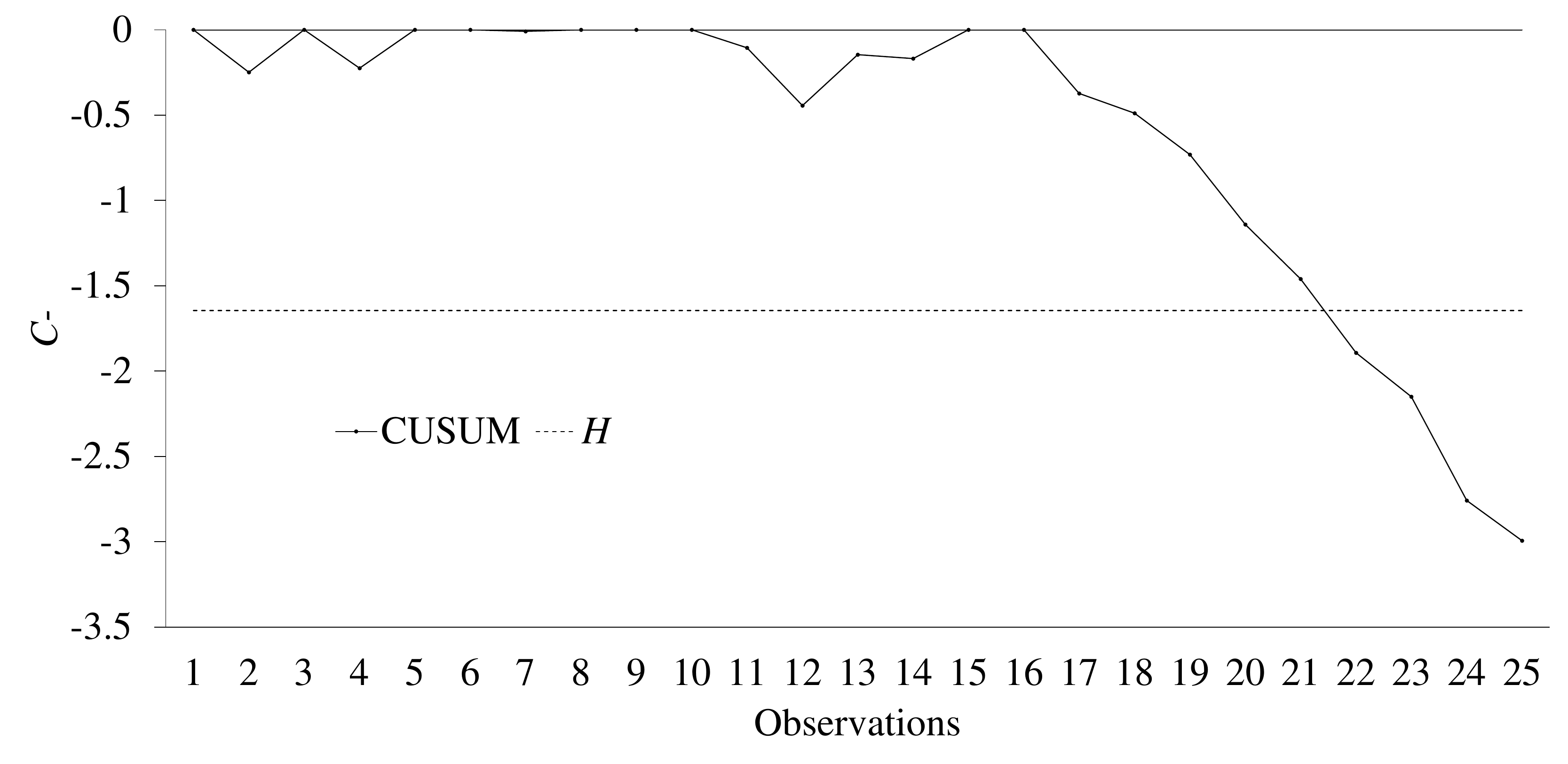}
  \caption{Optimal CUSUM for variance on sequential normal scores.}
  \label{fig:timesub2}
\end{subfigure}
\caption{Optimal CUSUM chart for subgroup variance with
  k = 0.793 and H = -1.645 was used on sequential normal scores obtained from service
times obtained from the bank example data set.}
\label{fig:time}
\end{figure}

\subsection{Individual observations with known quantile $\theta$}
\label{subsec:model3}
Let $X_1, X_2,\ldots, $ be a sequence of independent identically distributed random
variables with a continuous distribution function $F(x)$. In contrast to
Section~\ref{subsec:model1} we will assume that a quantile $\theta$ and its
corresponding probability $F(\theta)$ are known, or assumed to be known, such
as $\theta$ = median. Define the conditional sequential rank $R_{i|\theta}$ of
$X_i$ differently depending on if $X_i<\theta$ or if $X_i>\theta$. In particular,
if $X_i<\theta$ then the conditional sequential rank $R_{i|\theta}$ of $X_i$ is the
rank of $X_i$ relative only to the previous random variables that were less than
$\theta$, including $X_i$. On the other hand, if $X_i>\theta$ then the conditional
sequential rank $R_{i|\theta}$ of $X_i$ is the rank of $X_i$ relative only to the
previous random variables that were greater than $\theta$, including $X_i$.

Let $N_i^-$  be the number of random variables, of the first $i$ random variables in
the sequence, that are less than or equal to $\theta$, and let $N_i^+$ be the number of random
variables, of the first $i$ random variables in the sequence that are greater than
$\theta$. Then $N_i^- + N_i^+ = i$. For reasons explained in Section 3 we will use
\begin{align}
\label{eq:pitheta}
P_{i|\theta}=\begin{cases}
F(\theta) \frac{R_{i|\theta}- 0.5}{N_i^-} \quad &, if \quad X_i \leq \theta\\
F(\theta) + [1-F(\theta)] \frac{R_{i|\theta} - 0.5}{N_i^+} \quad &, if  \quad X_i > \theta
\end{cases}.
\end{align}
to estimate $F(X_i|\theta)$. \textit{Conditional sequential normal scores} are obtained from $P_{i|\theta}$
using $Z_{i|\theta} = \Phi^{-1} \left( P_{i|\theta} \right)$ where $\Phi^{-1}$ stands
for the inverse cumulative standard normal distribution function. As will be proven in
Section \ref{sec:theory}, the sequence $\left\{ Z_{i|\theta}:i=1,2,... \right\}$
consists of mutually independent asymptotically (as $i$ gets large) standard normal
random variables.

\subsubsection{Concrete strength example}
\label{subsubsec:concrete}
From \cite{AICHOUNIcontrol} we obtain a data set of compressive strength for
ready mixed concrete. The compression strength of 22 samples of concrete with a
target specification of 350 Kgf/cm$^2$ were measured over a period of 30 days.
As shown by \citeauthor{AICHOUNIcontrol}, observations do not fit
a normal distribution. They addressed the situation of non-normality by
using a Johnson transformation. However, the use of a transformation in small data
sets comes with the risk of over-fitting and the selection of a mistaken
transformation function. By using conditional sequential normal scores, this risk is avoided. Nevertheless,
\citeauthor{AICHOUNIcontrol} did not evaluate
whether sampled measurements achieved the target value, they tested only for isolated
changes (where the parametric scenario might be the best approach). To continue
with the analysis, we evaluate if the compressive strength is
actually significantly larger than the median target. By defining the nominal value of 350 as the target process median,
the analysis can be carried out by using conditional sequential normal scores from
equation \eqref{eq:pitheta}.
Here, it is assumed that under the null distribution $F(350)=0.5$.
Figure \ref{fig:concrete} shows the results after evaluating the conditional sequential normal
scores with a CUSUM
chart that assumes normal standard observations. From \cite{qiu2013introduction},
we obtained that an allowance of $k=0.25$ and
a control limit $H=5.597$ used as parameters of a one-sided tabular CUSUM
chart achieve an in-control performance of 200 in terms of average run length
($ARL_0$). Even though we are monitoring a null median of 350, and this information is incorporated
into the sequential normal scores, the monitored scores still have a mean value of zero
and their behavior can be approximated with a standard normal distribution.
As can be seen, an alarm is signaled at sample 21, which indicates that
the mixture is providing more strength than nominal specification.
Because sequential normal scores provide a conservative approximation of the standard normal
distribution, with a slightly smaller variance (see Subsection~\ref{subsec:SNSvar} for details),
the true $ARL_0$ is at most the one specified by the CUSUM setup.

\begin{figure}
\centering
\begin{subfigure}{.5\textwidth}
  \centering
  \includegraphics[width=1\linewidth]{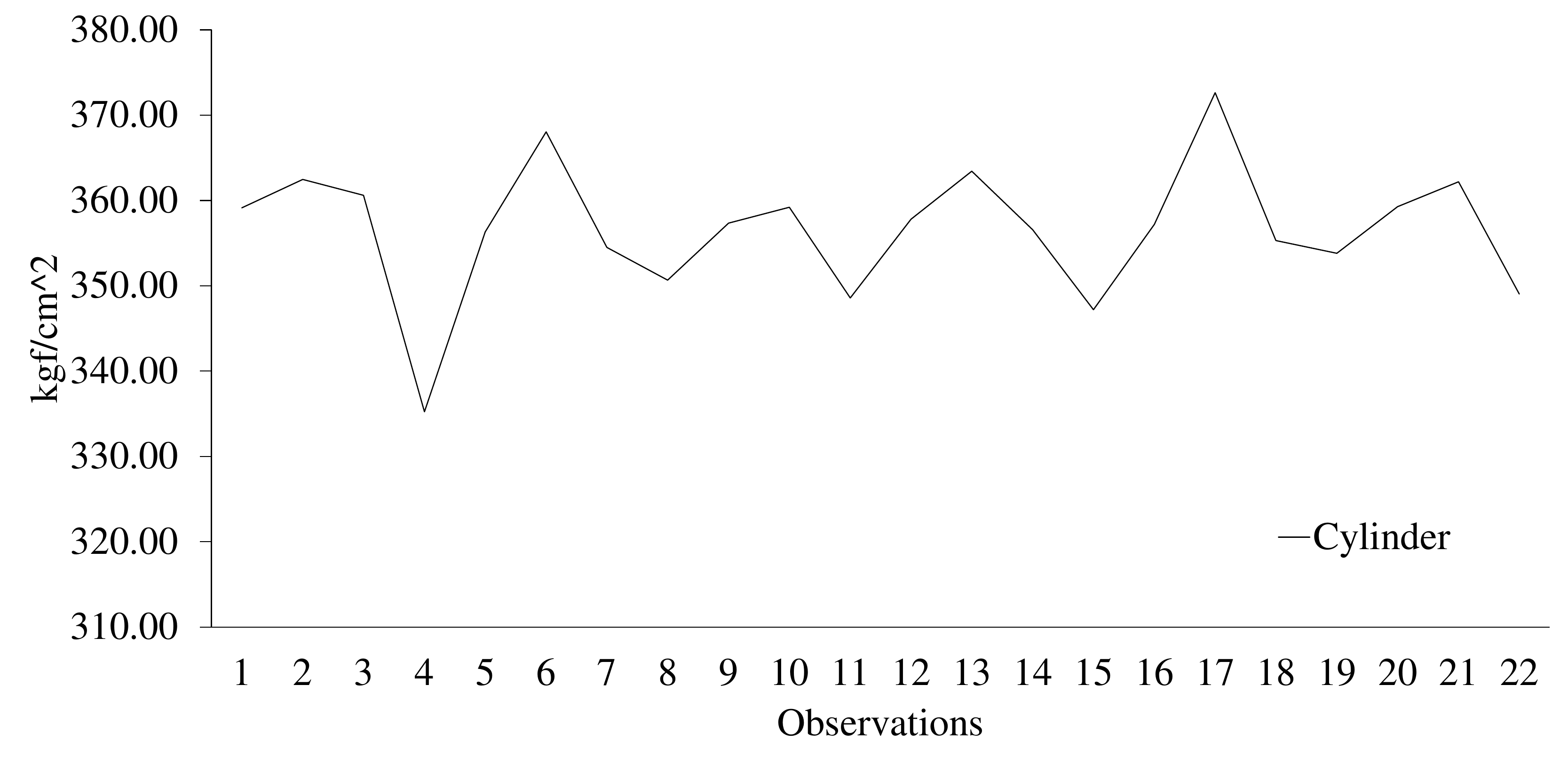}
  \caption{Average compressive strength measurements.}
  \label{fig:concretesub1}
\end{subfigure}%
\begin{subfigure}{.5\textwidth}
  \centering
  \includegraphics[width=1\linewidth]{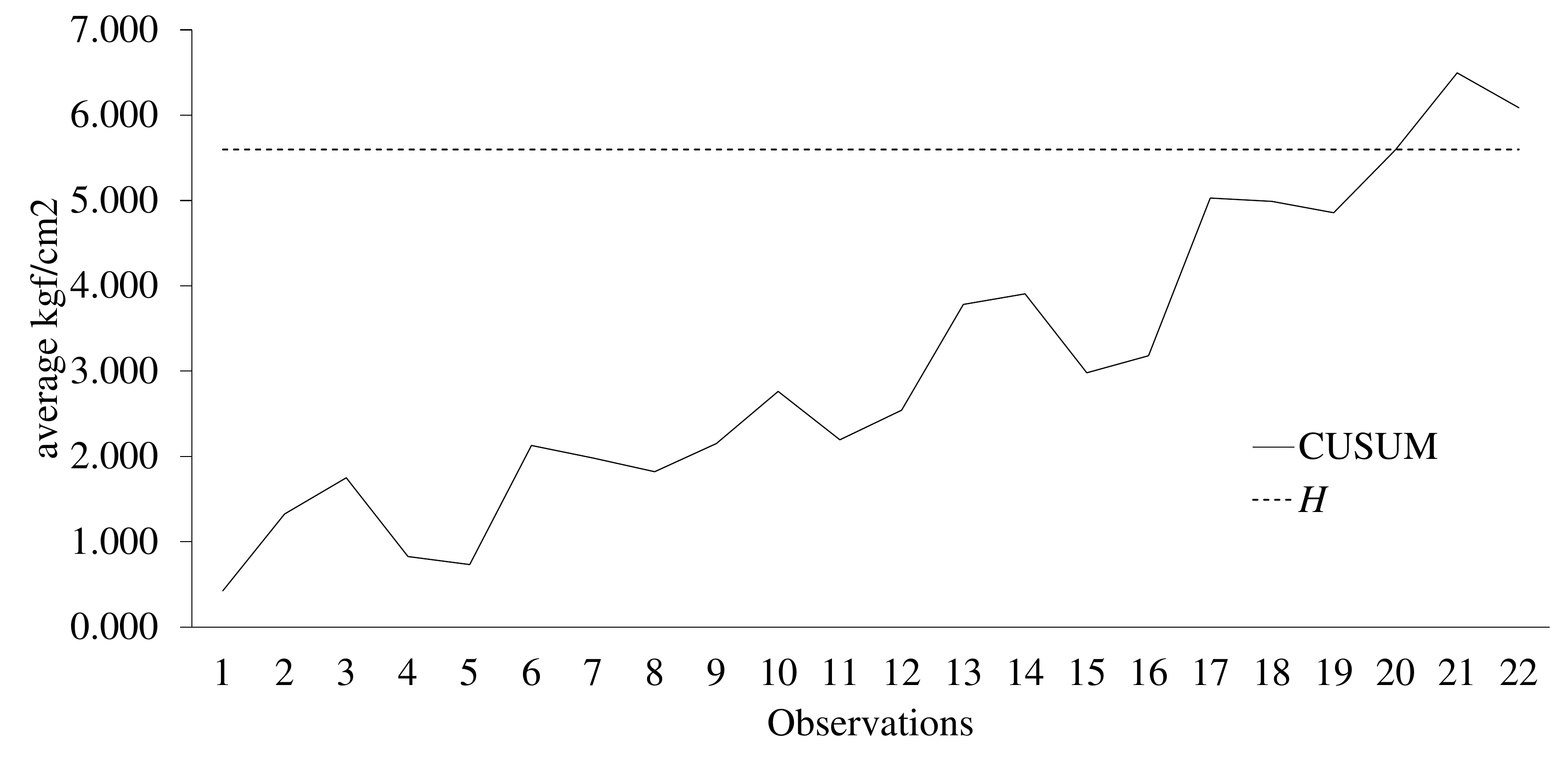}
  \caption{CUSUM chart on sequential normal scores.}
  \label{fig:concretesub2}
\end{subfigure}
\caption{CUSUM obtained from sequential normal scores of compressive strength for Ready
Mixed Concrete (Kgf/cm$^2$). The average was used as an individual observation.
An allowance of $k = 0.25$ and a decision interval of $H = 5.597$ were used.}
\label{fig:concrete}
\end{figure}

\subsection{Batched observations with known quantile $\theta$}
\label{subsection:model4}
Let $\{ X_{i,j}:i=1,\hdots,;j=1,\hdots,m \}$ be a sequence of independent
identically distributed random variables with a continuous distribution function
$F(x)$. The second subscript refers to the fact that these random variables are grouped
into batches (samples) of size m. In contrast to Section \ref{subsec:model2} we
will assume that a quantile $\theta$  and its corresponding probability $F(\theta)$
are known, or assumed to be known, such as $\theta$ = median. For the first batch
($i = 1$) the random variables that are less than $\theta$ are ranked relative to
only the other random variables in the first batch that are less than $\theta$, and the ranks
of the random variables that are greater than $\theta$ are ranked relative to only the
other random variables in the first batch that are greater than $\theta$. These conditional ranks
are denoted by $R_{1,j|\theta}$  for $j=1,\ldots, m$. For all subsequent batches define
the conditional sequential rank $R_{i,j|\theta}$ of $X_{i,j}$ differently depending on
if $X_{i,j} \leq\theta$ or if $X_{i,j}>\theta$ . In particular, if $X_{i,j} \leq\theta$,
then the conditional sequential rank $R_{i,j|\theta}$ of $X_{i,j}$ is the rank of
$X_{i,j}$ relative to only the random variables that were less than $\theta$ in the
previous batches, including $X_{i,j}$, but no other random variables in the same batch i.
On the other hand, if $X_{i,j}>\theta$ then the conditional sequential rank
$R_{i,j|\theta}$ of $X_{i,j}$ is the rank of $X_{i,j}$ relative to only the previous
random variables that were greater than $\theta$ in the previous batches, including
$X_{i,j}$ but no other random variables from the same batch i.

Let $N_1^-$ be the number of random variables of the first batch in the sequence,
that are less or equal than $\theta$ and $N_1^+$ be the number of random variables
of the first batch that are greater than $\theta$. For $i>1$, let $N_i^-$  be the 
number of random variables, of the first $i-1$ batches of
random
variables in the sequence, that are less than $\theta$, and let $N_i^+$ be the number
of random variables of the first $i-1$ batches of random variables in the
sequence that
are greater than $\theta$. Then, $N_1^- + N_1^+ = m$ and $N_i^- + N_i^+ = (i-1)m$ for $i>1$. 
For reasons explained in
Section~\ref{sec:theory} we will use $P_{1,j|\theta}$ for $i=1$ and $P_{i,j|\theta}$
for $i>1$, as
\begin{align}
\label{eq:P1jtheta}
P_{1,j|\theta}= \begin{cases}
F(\theta) \frac{R_{1,j|\theta} - 0.5}{N_1^-} \quad &, \quad X_{1,j} \leq \theta\\
F(\theta)+(1-F(\theta))\frac{R_{1,j|\theta} - 0.5}{N_1^+} \quad &, \quad X_{1,j}>\theta
\end{cases},
\end{align}
and for $i \geq 2$
\begin{align}
\label{eq:Pijtheta}
P_{i,j|\theta}= \begin{cases}
F(\theta) \frac{R_{i,j|\theta} - 0.5}{N_i^- +1} \quad &, \quad X_{i,j} \leq \theta\\
F(\theta)+(1-F(\theta))\frac{R_{i,j|\theta} - 0.5}{N_i^+ +1} \quad &, \quad X_{i,j}>\theta
\end{cases}.
\end{align}

\textit{Conditional sequential normal scores} are obtained from $P_{i,j|\theta}$
using $Z_{i,j|\theta}=\Phi^{-1} \left(P_{i,j|\theta}\right)$ where $\Phi^{-1}$ stands
for the inverse cumulative standard normal distribution function. As proven in
Section~\ref{sec:theory}, the sequence
$\left\{ Z_{i,j|\theta}:i\geq 2;j=1,\ldots,m \right\}$ consists of mutually
independent asymptotically (as $i$ gets large) standard normal random variables.

\subsubsection{GPA example}
From \cite{bakir2010monitoring} we obtained data that consists of GPA results from
management major students of the Department of Business Administration at
Alabama State University. Measurements were taken from the period of Spring 2005
through Spring 2009. The research showed that GPAs maintained a desired target
median level of 2.600 but they were significantly below the higher target
median of 2.800, which represented 70\% of the maximum score of 4 points. The data
presented problems of satisfying the assumption of normality, and a nonparametric control
chart based on signed ranks was proposed by the authors. The original data is plotted
in Figure \ref{fig:GPAsub1}. Using equation \eqref{eq:P1jtheta} and \eqref{eq:Pijtheta}, observations
are transformed into conditional sequential normal scores that, in turn, are evaluated using an EWMA statistic
\begin{align}
\label{eq:ewma}
U_i = \lambda \bar{Z_i}+(1-\lambda)U_{i-1}
\end{align}
where $U_0 = 0$, and $\bar{Z_i}$ is the average of $Z_{i,j|\theta}$; $j=1,\ldots,m$; as seen in \cite{qiu2013introduction}. Significance is
achieved when the plotted statistic surpasses the limits
\begin{align}
\label{eq:ewlimits}
\pm \rho \sqrt{\frac{\lambda}{2-\lambda}[1-(1-\lambda)^{2i}]}\frac{ \sigma}{\sqrt{m}}.
\end{align}
Here, $\sigma = 1$, the variance of the standard normal distribution approximated by
the sequential normal scores. Following the \verb|spc| package in \verb|R|, the
\verb|xewma.arl| function is used in a calibration process to obtain 
the value of parameter $\rho = 2.714$ with a convenient $\lambda = 0.1$ that approximate
a zero-start
$ARL_0 \approx 370.0$ with variable control limits. Results, as seen in
Figure~\ref{fig:GPAsub2}, show a
significant difference between observed data and the target value of 2.8.
Even though the original analysis evaluated batches individually, results
are consistent with the overall analysis, as a signal is triggered first at
the second sample and continues after the fourth.

\begin{figure}
\centering
\begin{subfigure}{.5\textwidth}
  \centering
  \includegraphics[width=1\linewidth]{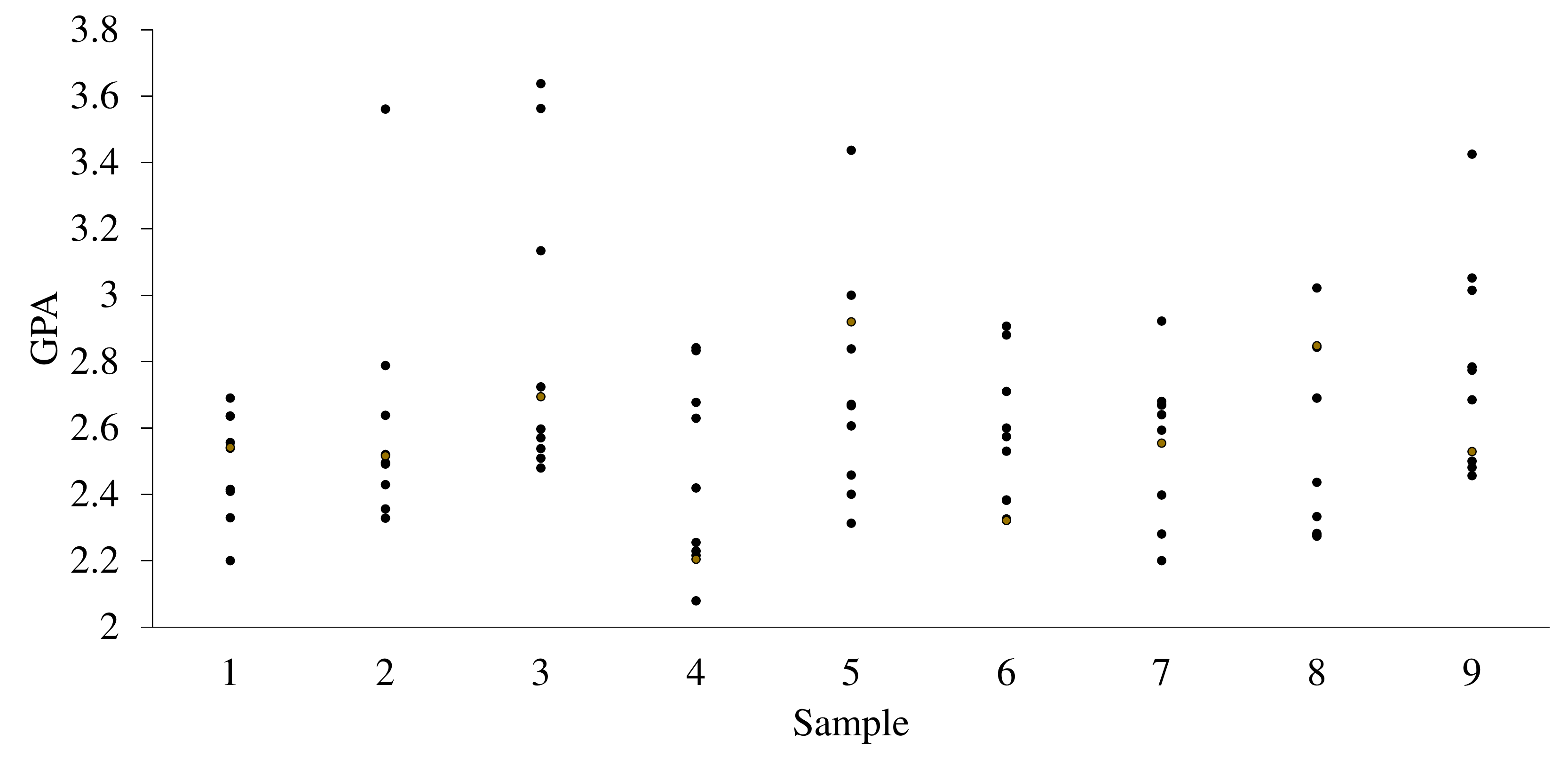}
  \caption{Original GPA measurements.\\~}
  \label{fig:GPAsub1}
\end{subfigure}%
\begin{subfigure}{.5\textwidth}
  \centering
  \includegraphics[width=1\linewidth]{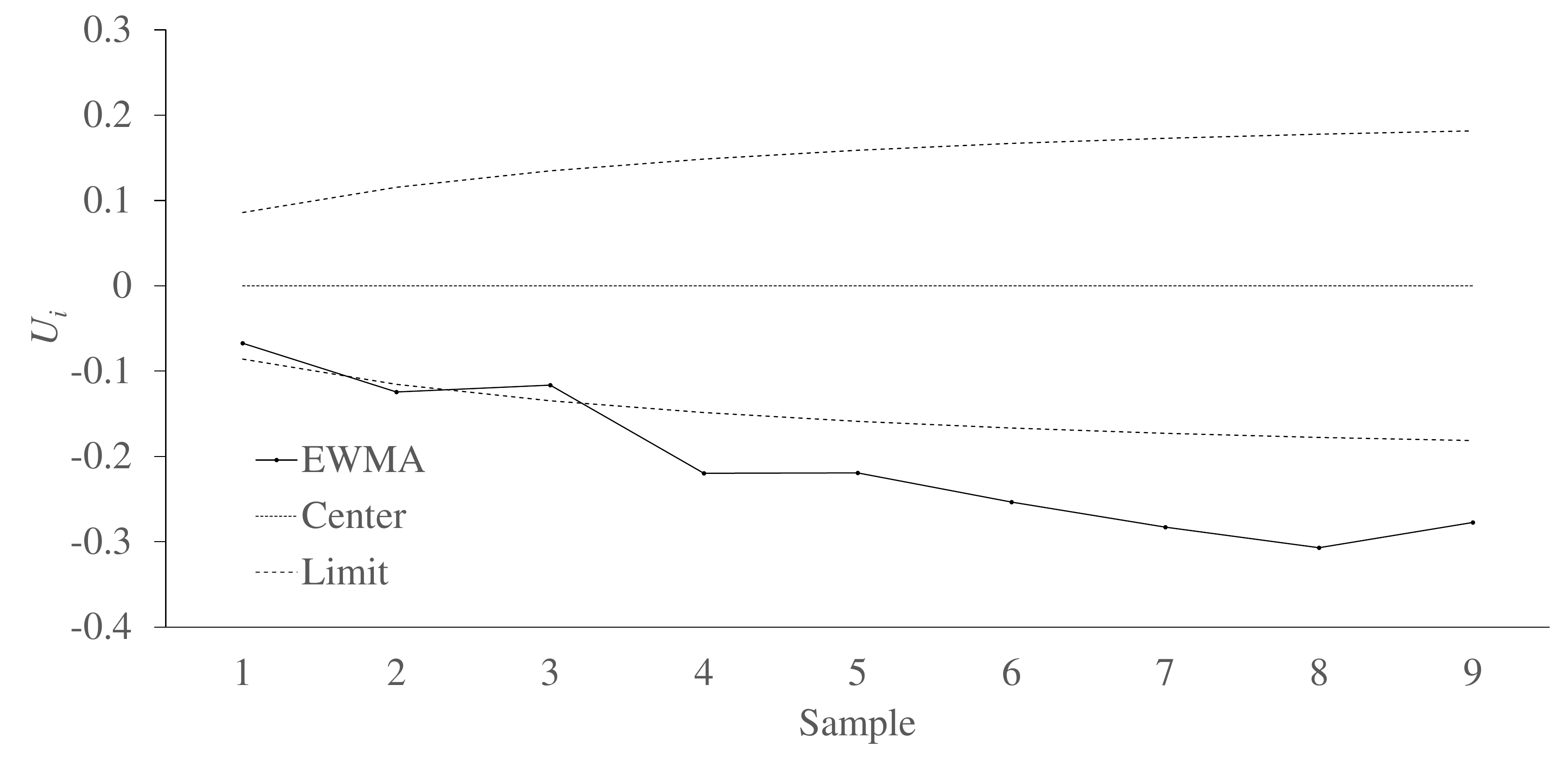}
  \caption{EWMA for subgroup observations on sequential normal scores.}
  \label{fig:GPAsub2}
\end{subfigure}
\caption{EWMA control chart applied over sequential normal scores from the GPA data set example with
a null median of 2.8.}
\label{fig:GPA}
\end{figure}

\subsection{Application remarks}
Sequential normal scores extend the use of parametric procedures to deal with observations from an
unknown distribution. However, practitioners should be aware that such a transformation
is only appropriate to deal with sustained changes. If isolated changes are a concern,
as might be the case for a test for outliers, a parametric approach might be the best option.

Big data users and researchers might find the sequential normal scores transformation appealing
due to its asymptotic
behavior and the reduced computational effort required for its implementation.
On one hand, as seen in Section \ref{sec:theory}, the transformation becomes
close to an exact
quantile transformation into normality. When a large data set is available, model
based approaches usually fail to fit the precise observed behavior. Nonparametric approaches,
such as sequential normal scores, are more likely to provide null distributions with a better fit to the
true unknown probability law behind the statistic used to evaluate the data.
On the other hand, by following a sequential ranking scheme that avoids re-ranking
previous observation, a large amount of computational time is saved. For instance,
Figure~\ref{fig:computertime}, illustrates the time, in seconds,
it takes for three nonparametric statistics used in statistical process
control literature, including sequential normal scores, to be calculated once, after a number of
observations is available from a data stream. We considered the statistic of
\cite{hawkins2010nonparametric}, which is a change-point statistic based on Mann-Whitney
statistic;  and the statistic of \cite{chowdhury2015distribution},
where a reference sample is evaluated against a monitored sample using
the Lepage statistic. All measurements were carried using
a Hewlett-Packard PC, model 6200 PRO SFF with an Intel Core I3 2120 processor, 500GB
of hard drive, 3GB for memory and Windows 7 Pro.
It can be seen that the statistic of \cite{hawkins2010nonparametric} could not be
computed for sample size bigger than $10^4$, whereas the  statistic
of \cite{chowdhury2015distribution} and sequential normal score statistic presented
good performance in terms of computational time for data streams
until $10^7$. In each scenario the computational time of sequential normal scores is at least  
10 times faster than \citeauthor{chowdhury2015distribution}'s [\citeyear{chowdhury2015distribution}] statistic.
\begin{figure}
\centering
\includegraphics[width=1\linewidth]{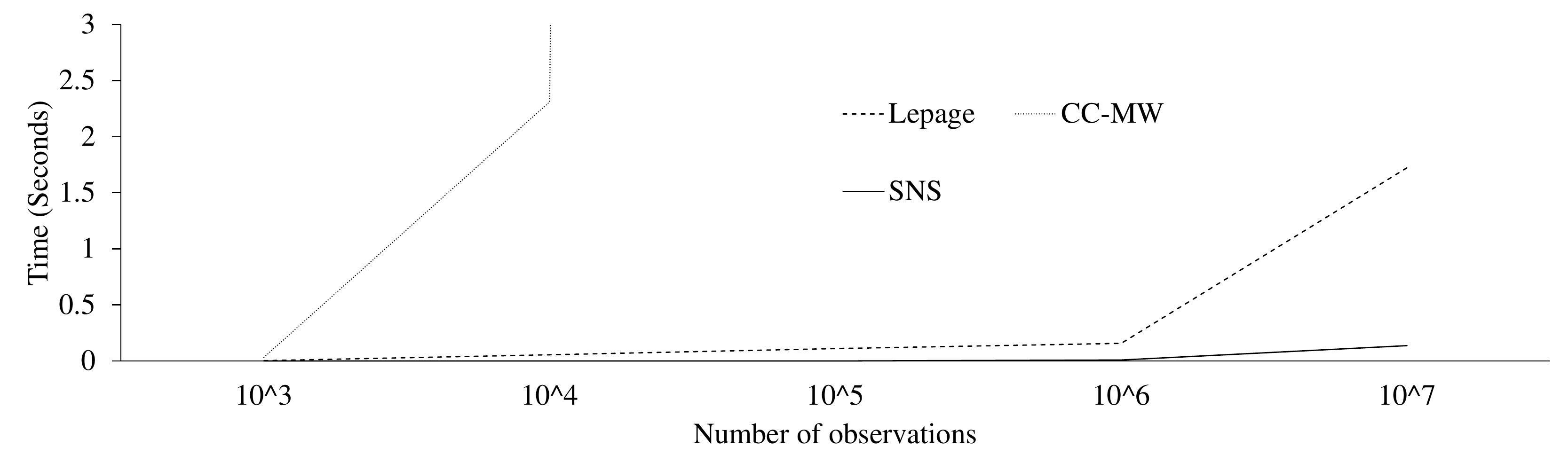}
\caption{Computer times to compute a charting statistic over different number of observations
from \cite{hawkins2010nonparametric} (CC-MW), \cite{chowdhury2015distribution} (Lepage), and
the sequential normal scores (SNS).}
\label{fig:computertime}
\end{figure}

\section{Theoretical framework}
\label{sec:theory}

\subsection{The mean of sequential normal scores}
Consider a sequence $X_1,X_2,...,$ of independent random variables identically
distributed according to a continuous distribution function $F(x)$.
Note that $F(X_i)$ is a uniform $(0, 1)$ random variable for all $i\geq 1$.
Sequential ranks $R_i=R(X_i)$ are defined by \cite{parent1965sequential} as $R_1=1$ and, for $i\geq 2$ as

\begin{align}
\label{eq:Ri}
R_i=\sum_{j=1}^{i-1} I(X_j\leq X_i) +1.
\end{align}

Note that the indicator variables $I(X_j \leq X_i)$ are Bernoulli random variables for $1\leq j\leq i-1$, with mean $1/2$, variance $1/4$, and covariance $1/12$.

Consider $P_i=P(X_i)$ as an estimator of $F(X_i)$, where
\begin{align}
P_i=\frac{(R_i -1+a)}{i-1+b}\quad &, \quad b>0;\quad 0<a<b.
\end{align}
Thus, $0<P_i<1$ and the sequential normal score $Z_i=\Phi^{-1}(P_i)$ is a well-defined normal score, where $\Phi^{-1}$ is the inverse of the standard normal distribution function

We would like to choose $a$ and $b$ so that the mean and variance of $Z_i$ are 0 and 1 respectively. The mean of $Z_i$ equals 0 if and only if the mean (and median) of $P_i$ is 0.5,
due to the quantile preserving property of monotonic transformations. The mean of $P_i$ is,
from equation \eqref{eq:Ri}
\begin{align}
E(P_i)=\frac{\frac{i-1}{2} +a}{i-1+b}
\end{align}
which  equals 0.5 if and only if $b=2a$. The variance of $Z_i$ is a function of $b$, which is discussed in next subsection.

\subsection{The variance of sequential normal scores}
\label{subsec:SNSvar}
The variance of $P_i$ can be easily found to be

\begin{align}
\label{eq:VarSa}
Var(P_i)=\frac{\frac{i-1}{4}+\frac{(i-1)(i-2)}{12}}{(i-1+b)^2}=\frac{i^2-1}{12(i-1+b)^2}
\end{align}
which is a function of b, and equals 1/12 (the variance of $F(X)$) if and only if $b=1-i+\sqrt{i^2-1}$. As $i$ gets large, $b$ approaches 1, which suggests using $b=1$ for finite sample sizes.

Unfortunately the variance is not preserved in the normal scores transformation, so we are forced to use numerical approximation methods to find values of b that result in $Var(Z_i)=1$. These results are summarized in Table \ref{tab:std for b}.

Table \ref{tab:std for b} shows that the actual value of $b$ to obtain $Var(Z_i)=1$ increases
from $b=0.465$ for $i=2$ to $b=0.915$ for $i=5000$. The effect of using $b=1$ as an
approximation is slight (less than $2\%$ error for the standard deviation of $Z_i$)
for $i>31$, and the approximation
\begin{align}
\label{eq:baprox}
b = 0.824-\frac{0.792}{i}
\end{align}
results in less than a $1\%$ error for $i > 2$ ($3.7\%$ error for $i = 2$ ).
If $b = 1$ is used for all values the result is a conservative estimate for the standard
deviation of the sequential normal scores $Z_i$, so that is what we used in our examples.
A slight increase in power results from using the approximation for $b$ when $i<31$.
Nevertheless, the increased power was not big enough to change the decisions obtained from the
particular numerical examples shown in Section \ref{sec:SNS}.

This approximation problem has been previously addressed. For example, 
\cite{van1952order} used $a=1$, $b=2$; \cite{blom1954transformations} used 
$a=5/8$, $b=1.25$; \cite{tukey1962future} used $a= 2/3$, $b=4/3$; and 
\cite{bliss1956rankit} used $a=1/2$, $b=1$ which they called a “rankit”. 
An empirical study by \cite{solomon2009impact} compared these four approximations 
and concluded that ``Rankit emerged as the most accurate method among all 
sample sizes and distributions, thus it should be the default selection for score 
normalization in the social and behavioral sciences''(p.448). Note that all four 
approximations preserve a mean of 0 for the normal scores because $b = 2a$. However all 
four approximations result in a variance less than 1.0 which we addressed in this section.

\begin{table}[H]
\centering
\caption{Standard deviation of $Z_i$ when $b=1$, and values of $b$ that give 1.0 for the standard deviation of $Z_i$. Also given is the standard deviation of $Z_i$ when the approximation for $b$ from equation \eqref{eq:baprox} is used.}
\label{tab:std for b}
\begin{tabular}{ccccc}
\hline
i    & $\sigma$ when $b = 1$ & \begin{tabular}[c]{@{}l@{}}Value of $b$ to\\   obtain $\sigma = 1$\end{tabular} & \begin{tabular}[c]{@{}l@{}}Value of $b$ using \\   approximation\end{tabular} & \begin{tabular}[c]{@{}l@{}}Value of $\sigma$ using $b$\\   from equation~\eqref{eq:baprox}\end{tabular} \\ \hline
2    & .675         & .465                                                                   & .428                                                                        & 1.037                                                                         \\
3    & .790         & .565                                                                   & .560                                                                        & 1.004                                                                         \\
4    & .844         & .618                                                                   & .626                                                                        & .996                                                                          \\
5    & .876         & .652                                                                   & .666                                                                        & .994                                                                          \\
10   & .938         & .727                                                                   & .745                                                                        & .995                                                                          \\
20   & .969         & .778                                                                   & .784                                                                        & .999                                                                          \\
30   & .979         & .800                                                                   & .798                                                                        & 1.000                                                                         \\
31   & .980         & .802                                                                   & .798                                                                        & 1.000                                                                         \\
32   & .980         & .803                                                                   & .799                                                                        & 1.000                                                                         \\
100  & .994         & .847                                                                   & .816                                                                        & 1.001                                                                         \\
1000 & .999         & .896                                                                   & .823                                                                        & 1.001                                                                        \\
5000 & 1.000        & .915                                                                   & .824                                                                        & 1.000                                                                        \\ \hline
\end{tabular}
\end{table}

\subsection{Include a priori information about a known quantile}

Let $X_1,X_2,\ldots,$ be a sequence of independent observations,
and $F(\cdot)$ corresponds to their
common cumulative null distribution. Also, assume that $F(\theta)$ and
$\theta$ are known, where the latter is a constant.
(To facilitate notation we'll be using
$F(a)$ and $P(X \leq a)$, without the
subindex in the variable $X$, interchangeably, to express the evaluation
of the cumulative distribution function at a constant a under the null
hypothesis, and we will suppress the fact that
$F(\theta)$ is known even though it is implicit
in all of the probabilities of this section).

Given a constant $a$, a conditional sequential normal scores statistic can be constructed by noting that
\begin{align}
\label{eq:p0}
P(X \leq a) = P(X \leq a| X \leq \theta)P(X \leq \theta) + P(X \leq a| X > \theta)P(X > \theta)
\end{align}
where the $P(X \leq a)$ can be estimated by using
\begin{align}
\label{eq:p1}
\hat{P}(X_i \leq a| X_i \leq \theta) = \frac{\sum_{j=1}^{i-1} I(X_j \leq a)I(X_j \leq \theta) + 0.5}{\sum_{j=1}^{i-1} I(X_j \leq \theta) + 1},
\end{align}
and
\begin{align}
\label{eq:p2}
\hat{P}(X_i \leq a| X_i > \theta) = \frac{\sum_{j=1}^{i-1} I(X_j \leq a)I(X_j > \theta) + 0.5}{\sum_{j=1}^{i-1} I(X_j > \theta) + 1}
\end{align}
in equation \eqref{eq:p0}. Equations \eqref{eq:p1} and \eqref{eq:p2} are maximum likelihood
estimators with biasing constants 0.5 and 1 defined in previous section.
Hence, using the fact that $P(X \leq \theta)$ is known, a cumulative probability can be
better estimated by
\begin{align}
\label{eq:probupdate}
\hat{P}(X \leq X_i) = \begin{cases}
P(X \leq \theta) \frac{\sum_{j=1}^{i-1} I(X_j \leq X_i)I(X_j \leq \theta) + 0.5}{\sum_{j=1}^{i-1} I(X_j \leq \theta) + 1} \quad &, \quad X_i \leq \theta\\
P(X \leq \theta) + P(X > \theta) \frac{\sum_{j=1}^{i-1} I(X_j \leq X_i)I(X_j > \theta) + 0.5}{\sum_{j=1}^{i-1} I(X_j > \theta) + 1} \quad &, \quad X_i > \theta
\end{cases}
\end{align}
Hence, conditional sequential normal scores are then defined by evaluating $\Phi^{-1}(\hat{P}(X \leq X_i))$.

\subsection{Sequence of Independent Statistics}
The following Proposition states the independence of the sequence
$\left\{P_i:1,\hdots,n \right\}$ and the asymptotic distribution of
the sequence $\left\{Z_i:1,\hdots,n \right\}$, of model in Section~\ref{subsec:model1}.

\begin{mypropo*}
\label{propo:Pi}
Let $X_1,X_2,...,X_n$ be a sequence of i.i.d. random variables. Define $P_i$ and $Z_i$ as in Section \ref{subsec:model1}. Then,
\begin{enumerate}
  \item \label{itm:first} Series $\left\{P_i:1,\hdots,n \right\}$ are mutually independent random variables.
  \item \label{itm:second} $\left\{Z_i:1,\hdots,n \right\}$ are independent asymptotically standard normal random variables.
\end{enumerate}

\end{mypropo*}

\begin{proof}
 Since Theorem 1.1 in \cite{barndorff1963limit} states that the sequence of random variables
 $\left\{ R_i(X_i):i=1,2,\hdots \right\}$ are mutually independent,
 it follows by Theorem 4.6.12 in \cite{casella2002statistical}
 that $\left\{ P_i:i=1,2,\hdots \right\}$ are also mutually independent random variables.
 This in turn implies that the $Z_i$ are also mutually independent. By applying the Strong
 Law of Large Numbers, the Glivenko-Cantelli Theorem implies that $P_i$ converges uniformly
 to a uniform (0,1) random variable, and thus $Z_i$ converges to a standard normal random
 variable, as $i$ goes to infinity.
\end{proof}

It can be noted that the results of the Proposition apply to the models in
Sections~\ref{subsec:model2}, \ref{subsec:model3} and \ref{subsection:model4} as well.
The variations in the proof involve redefining
sequential ranks in each section, but otherwise are trivial and are omitted here.

\begin{mycoro*}
The results of the Proposition hold even if Pi,j is given by
equations~\eqref{eq:P1j} and \eqref{eq:Pij} in Section~\ref{subsec:model2}, or if $P_{i|\theta}$ is
given by equation \eqref{eq:pitheta}
in Section~\ref{subsec:model3}, or if $P_{ij|\theta}$ is given by
equations \eqref{eq:P1jtheta} and \eqref{eq:Pijtheta} in Section~\ref{subsection:model4}.
\end{mycoro*}

\section{Limited sample approximation}
Previous results indicate that sequential normal scores are independent, and they
approach to a standard normal random variable as the series grows to infinity.
However, in practice, practitioners are limited to a finite number of observations
and might wonder how sequential normal scores will perform with that restriction.
Because this paper introduces the concept of sequential normal scores for the first time,
the authors thought it would be beneficial to examine the behavior of sequential normal
scores in more detail.
Extensive Monte Carlo simulations were conducted to determine how well a standard
normal distribution approximates the exact distribution of sequential normal scores,
as defined in Section~\ref{subsec:model1}. Although an exact distribution can be obtained
by using the fact that every ordering of the usual ranks of a random sequence of
independent and identically distributed random variables is equally likely, and there
is a one-to-one correspondence between the usual ranks and the sequential ranks,
it was more convenient to use Monte Carlo sampling in this study.

First, the exact distribution of the $n$-th sequential rank $R_n$ is well known to be
the uniform distribution over the integers 1 to $n$. But, after converting to $P_n$ and
then to $Z_n$, how does the distribution of $Z_n$  compare with the standard normal
distribution? Figure~\ref{fig:NSsim} shows the comparison for various values of $n$
from 2 to 1000. It shows that the choice of using $b = 1$ for all values of $n$ does not
appear to make an appreciable difference, as opposed to tailoring values of $b$ to make
the variance closer to 1.0.

Second, the empirical distribution of the first $n$ sequential normal scores is compared
with the standard normal distribution by averaging 1000 empirical distribution functions,
with the results shown in Figure~\ref{fig:SNSsim}. These results show that the standard
normal quantiles in the tails tend to be larger (in absolute value) than the exact
quantiles, which will result in conservative tests, but the difference appears to be
negligible for $n$ greater than 30. This agrees with our comparison of exact variances
with the variance of the standard normal distribution in Section~\ref{subsec:SNSvar}.
Also, these exact distributions reveal a relatively large probability at $x = 0$,
but this jump disappears almost completely as $n$ reaches 100 or more.

Finally, a single random sequence is evaluated at various lengths in
Figure~\ref{fig:SNSpath}, and an Anderson-Darling goodness of fit is applied.
The resulting empirical distribution functions show some divergence from the standard
normal distribution in the middle values of $x$, but good agreement in the tails of
the distributions where a good approximation is more important. The resulting p-values
for this series are significant at $n = 100$ and $n = 300$, but the significance
disappears for larger values of $n$.

\begin{figure}[h]
	\centering
	\includegraphics[width=\textwidth]{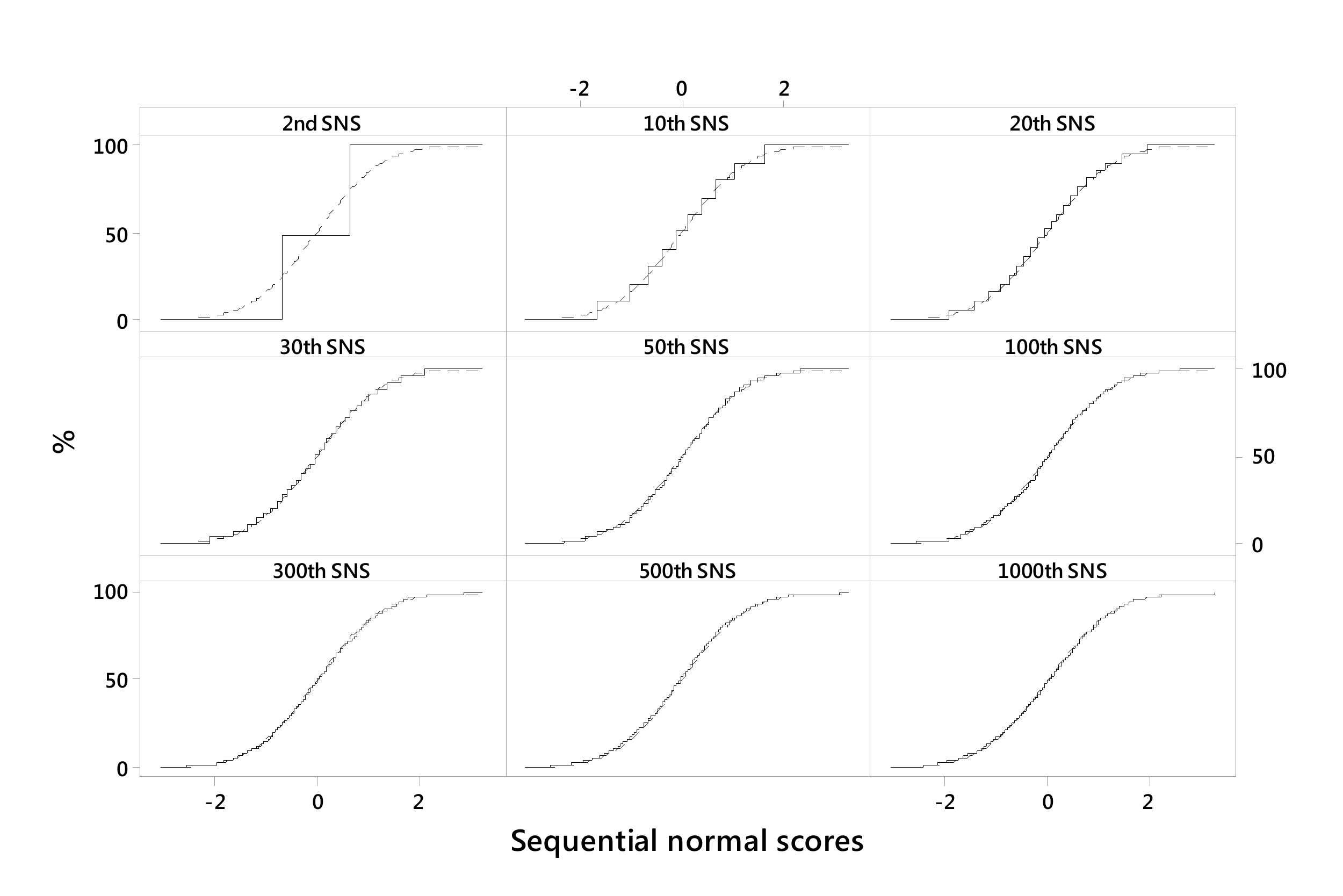}
    \caption{The exact distribution (solid line) of the $n$-th sequential normal score
    at different values of $n$, and the cumulative standard normal distribution (dotted line).}
    \label{fig:NSsim}
\end{figure}

\begin{figure}[h]
	\centering
	\includegraphics[width=\textwidth]{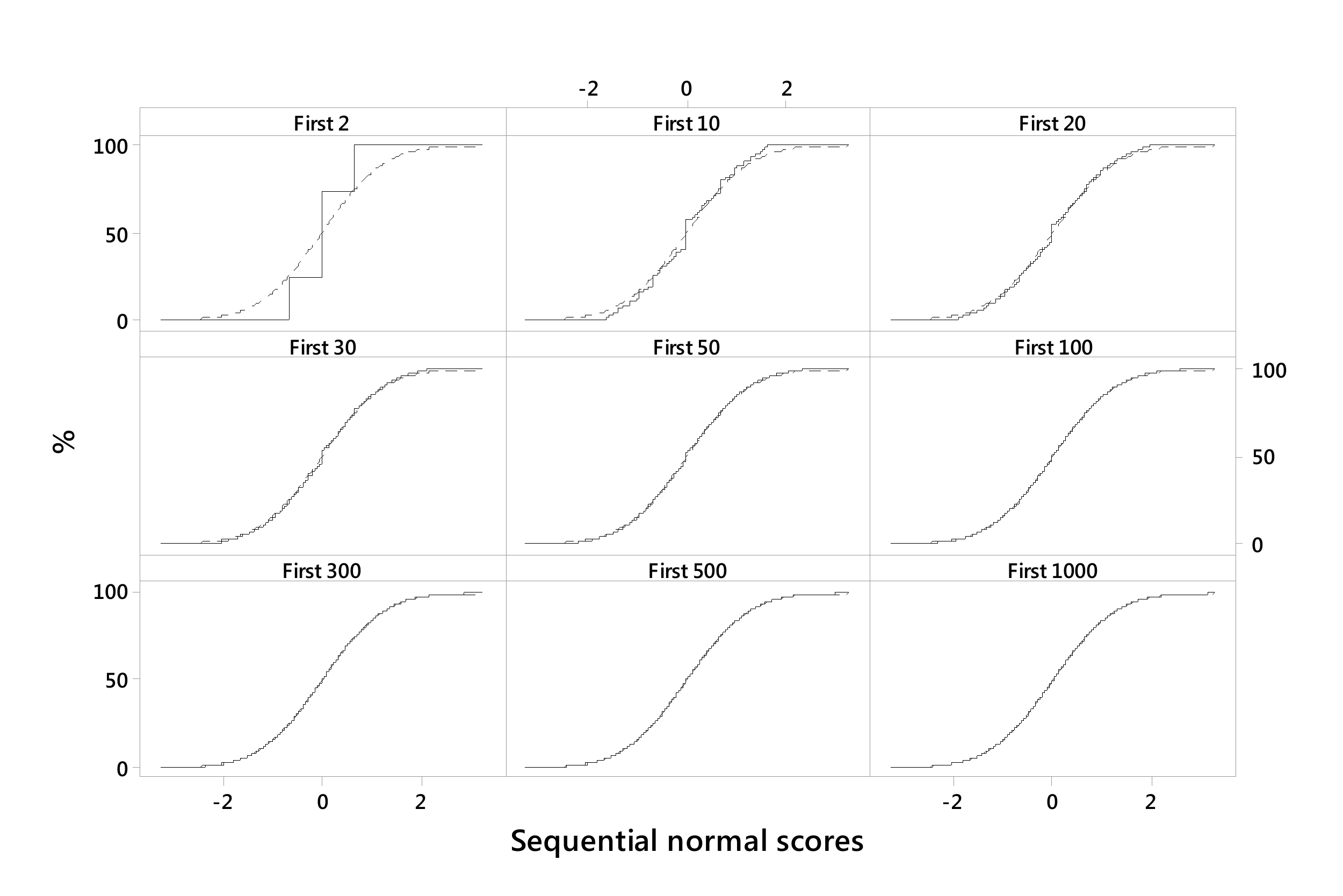}
    \caption{The expected value of the empirical distribution function (solid line)
    for different sequence lengths $n$, and the cumulative standard normal distribution
    (dotted line).}
    \label{fig:SNSsim}
\end{figure}

\begin{figure}[h]
	\centering
	\includegraphics[width=\textwidth]{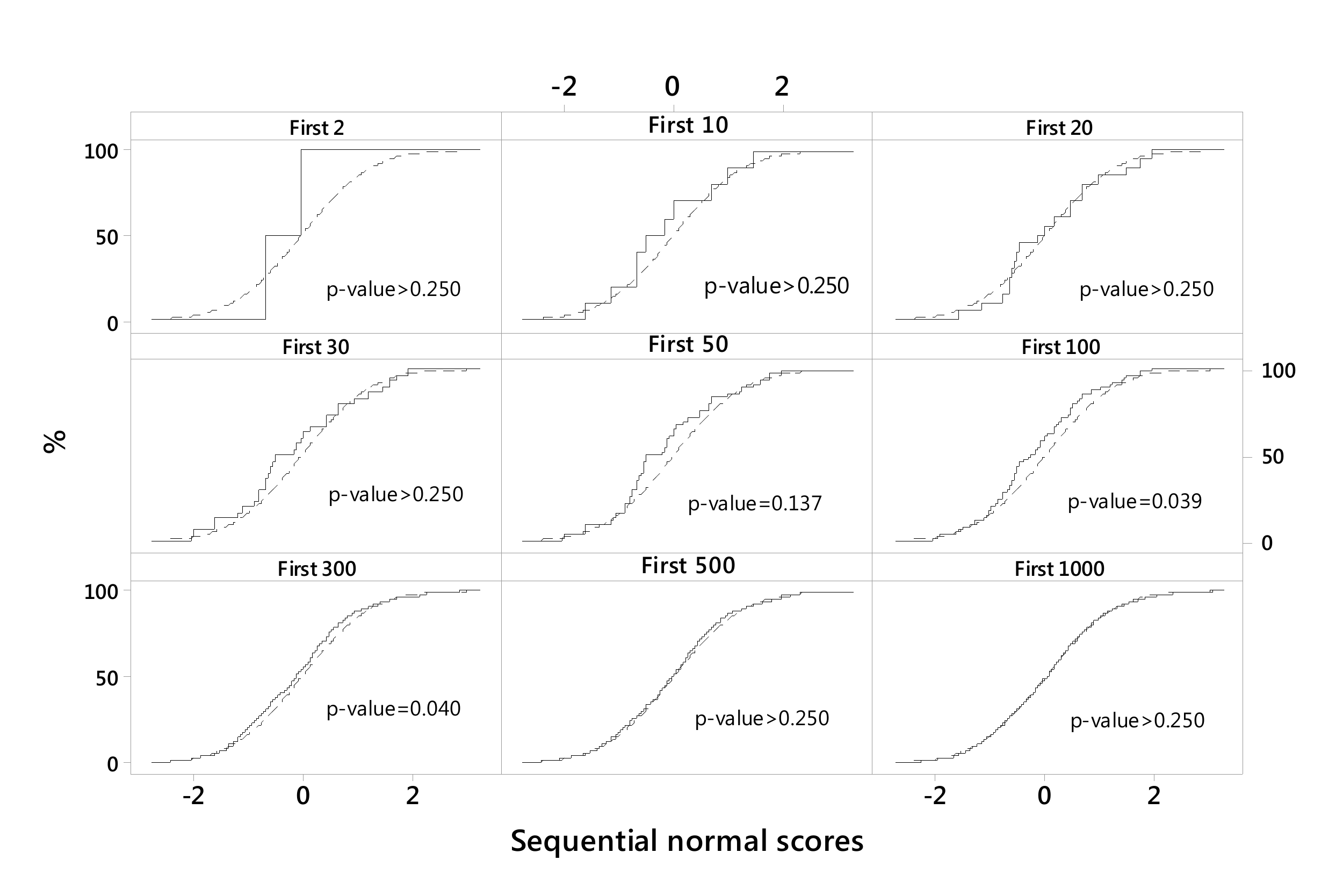}
	\caption{Empirical distributions of a sample path of sequential normal scores,
	the cumulative normal standard distribution and the p-value obtained from
	an Anderson-Darling test for $N(0,1)$ at different moments in time.} \label{fig:SNSpath}
\end{figure}

\section{Conclusions}
\label{sec:conclusions}
A new sequential statistic based on normal scores, named sequential normal scores,
was proposed as a link between parametric and nonparametric procedures that
are sequential in nature, such as control charts used in online
monitoring of data streams. The statistic extends the concept of sequential ranks into
a modified version of normal scores to obtain a sequence of
asymptotically standard normal and independent scores that can be analyzed with
existing procedures originally created to deal with normal and
independent observations. Four different versions of the statistic were presented
to address different situations with individual or batched observations, or when
a priori knowledge about a quantile exists or not. When sustained changes
are a concern, the applicability of the proposed transformation is illustrated with
the utilization of control chart on sequential normal scores from different databases
available in the
open literature. It is shown that sequential normal scores used with control
charts are capable of detecting changes in a process, requiring only i.i.d.
observations under the null hypothesis. Also, if a priori knowledge exists about
a process quantile under the null assumption, which might be the case when it is
desired that a process follows a target median, sequential normal scores are adapted to incorporate
existing information in such a way that control charts with sequential normal scores become sensitive
to detect changes at the very start of a monitoring.
Big data users might be willing to analyze their data streams with the sequential normal
scores transformation. If streams are very large, the fit to a normal distribution
provided by sequential normal scores might be, in many cases, arguably better,
than the fit one might find from parametric models--unless the true underlying
distribution is known. In addition, even though the sequential normal scores lack of complexity
makes them ``computer friendly'', if data streams are very large, it is easy to restrict
our comparisons to a moving window of reasonable size (e.g. 1000 or 5000) when determining
the sequential ranks. This sets the memory usage of a computer and the number of operations
to transform data constant.
The proposed statistic offers
a tool that makes nonparametric analysis readily available for practitioners and an
approach that can be used by researchers to deal with new problems as they become
available.

\bibliographystyle{plainnat}
\bibliography{Paper_SNS_20170509}

\begin{thebibliography}{25}
\providecommand{\natexlab}[1]{#1}
\providecommand{\url}[1]{\texttt{#1}}
\expandafter\ifx\csname urlstyle\endcsname\relax
  \providecommand{\doi}[1]{doi: #1}\else
  \providecommand{\doi}{doi: \begingroup \urlstyle{rm}\Url}\fi

\bibitem[Aichouni et~al.(2014)Aichouni, Al-Ghonamy, and
  Bachioua]{AICHOUNIcontrol}
M~Aichouni, AI~Al-Ghonamy, and L~Bachioua.
\newblock Control charts for non-normal data: Illustrative example from the
  construction industry business.
\newblock In \emph{Proceedings of the 16th International Conference on
  Mathematical and Computational Methods in Science and Engineering}, 24, pages
  71--76, Kuala Lumpur, Malaysia, 2014. WSEAS Press.

\bibitem[Bakir and McNeal(2010)]{bakir2010monitoring}
Saad~T Bakir and Bob McNeal.
\newblock Monitoring the level of students’ gpas over time.
\newblock \emph{American Journal of Business Education (AJBE)}, 3\penalty0
  (6):\penalty0 43--50, 2010.

\bibitem[Barndorff-Nielsen(1963)]{barndorff1963limit}
Ole Barndorff-Nielsen.
\newblock On the limit behaviour of extreme order statistics.
\newblock \emph{The Annals of Mathematical Statistics}, 34\penalty0
  (3):\penalty0 992--1002, 1963.

\bibitem[Barraza-Barraza(2015)]{barraza2015adaptive}
Diana Barraza-Barraza.
\newblock \emph{An adaptive ARX model to Estimate an Asset Remaining Useful
  Life}.
\newblock PhD thesis, Texas Tech University, 2015.

\bibitem[Bliss et~al.(1956)Bliss, Greenwood, and White]{bliss1956rankit}
CI~Bliss, Mary~L Greenwood, and Edna~Sakamoto White.
\newblock A rankit analysis of paired comparisons for measuring the effect of
  sprays on flavor.
\newblock \emph{Biometrics}, 12\penalty0 (4):\penalty0 381--403, 1956.

\bibitem[Blom(1954)]{blom1954transformations}
Gunnar Blom.
\newblock Transformations of the binomial, negative binomial, poisson and
  $\chi$ 2 distributions.
\newblock \emph{Biometrika}, 41\penalty0 (3/4):\penalty0 302--316, 1954.

\bibitem[Box(1957)]{box1957evolutionary}
George~EP Box.
\newblock Evolutionary operation: A method for increasing industrial
  productivity.
\newblock \emph{Applied Statistics}, pages 81--101, 1957.

\bibitem[Casella and Berger(2002)]{casella2002statistical}
George Casella and Roger~L Berger.
\newblock \emph{Statistical inference}, volume~2.
\newblock Duxbury Pacific Grove, CA, 2002.

\bibitem[Chowdhury et~al.(2015)Chowdhury, Mukherjee, and
  Chakraborti]{chowdhury2015distribution}
Shovan Chowdhury, Amitava Mukherjee, and Subhabrata Chakraborti.
\newblock Distribution-free phase ii cusum control chart for joint monitoring
  of location and scale.
\newblock \emph{Quality and Reliability Engineering International}, 31\penalty0
  (1):\penalty0 135--151, 2015.

\bibitem[Dodge and Romig(1929)]{dodge1929method}
Harold~French Dodge and HG~Romig.
\newblock A method of sampling inspection.
\newblock \emph{Bell System Technical Journal}, 8\penalty0 (4):\penalty0
  613--631, 1929.

\bibitem[Hawkins and Deng(2010)]{hawkins2010nonparametric}
Douglas~M Hawkins and Qiqi Deng.
\newblock A nonparametric change-point control chart.
\newblock \emph{Journal of Quality Technology}, 42\penalty0 (2):\penalty0 165,
  2010.

\bibitem[Hawkins and Olwell(2012)]{hawkins2012cumulative}
Douglas~M Hawkins and David~H Olwell.
\newblock \emph{Cumulative sum charts and charting for quality improvement}.
\newblock Springer Science \& Business Media, 2012.

\bibitem[Nectoux et~al.(2012)Nectoux, Gouriveau, Medjaher, Ramasso,
  Chebel-Morello, Zerhouni, and Varnier]{nectoux2012pronostia}
Patrick Nectoux, Rafael Gouriveau, Kamal Medjaher, Emmanuel Ramasso, Brigitte
  Chebel-Morello, Noureddine Zerhouni, and Christophe Varnier.
\newblock Pronostia: An experimental platform for bearings accelerated
  degradation tests.
\newblock In \emph{IEEE International Conference on Prognostics and Health
  Management, PHM'12.}, pages 1--8. IEEE Catalog Number: CPF12PHM-CDR, 2012.

\bibitem[Parent(1965)]{parent1965sequential}
E.A. Parent.
\newblock \emph{Sequential Ranking Procedures}.
\newblock PhD thesis, Department of Statistics, Stanford University., 1965.
\newblock URL \url{https://books.google.com.mx/books?id=pYFCAAAAIAAJ}.

\bibitem[Qiu(2013)]{qiu2013introduction}
Peihua Qiu.
\newblock \emph{Introduction to statistical process control}.
\newblock CRC Press, 2013.

\bibitem[Reynolds(1975)]{reynolds1975sequential}
Marion~R Reynolds.
\newblock A sequential signed-rank test for symmetry.
\newblock \emph{The Annals of Statistics}, 3:\penalty0 382--400, 1975.

\bibitem[Ross et~al.(2011)Ross, Tasoulis, and Adams]{ross_2011}
Gordon~J Ross, Dimitris~K Tasoulis, and Niall~M Adams.
\newblock Nonparametric monitoring of data streams for changes in location and
  scale.
\newblock \emph{Technometrics}, 53\penalty0 (4):\penalty0 379--389, 2011.

\bibitem[Shewhart(1931)]{shewhart1931}
Walter~A. Shewhart.
\newblock \emph{Economic control of quality of manufactured product}.
\newblock D. Van Nostrand Company, Inc., 1931.

\bibitem[Solomon and Sawilowsky(2009)]{solomon2009impact}
Shira~R Solomon and Shlomo~S Sawilowsky.
\newblock Impact of rank-based normalizing transformations on the accuracy of
  test scores.
\newblock \emph{Journal of Modern Applied Statistical Methods}, 8\penalty0
  (2):\penalty0 448--462, 2009.

\bibitem[Tukey(1962)]{tukey1962future}
John~W Tukey.
\newblock The future of data analysis.
\newblock \emph{The Annals of Mathematical Statistics}, 33\penalty0
  (1):\penalty0 1--67, 1962.

\bibitem[Van~der Waerden(1952)]{van1952order}
BL~Van~der Waerden.
\newblock Order tests for the two-sample problem and their power.
\newblock In \emph{Indagationes Mathematicae (Proceedings)}, volume~55, pages
  453--458. Elsevier, 1952.

\bibitem[Wald(1945)]{wald1945sequential}
Abraham Wald.
\newblock Sequential tests of statistical hypotheses.
\newblock \emph{The Annals of Mathematical Statistics}, 16\penalty0
  (2):\penalty0 117--186, 1945.

\bibitem[Wald and Wolfowitz(1948)]{wald1948optimum}
Abraham Wald and Jacob Wolfowitz.
\newblock Optimum character of the sequential probability ratio test.
\newblock \emph{The Annals of Mathematical Statistics}, 19:\penalty0 326--339,
  1948.

\bibitem[Yang and Arnold(2015)]{yang2015new}
Su-Fen Yang and Barry~C Arnold.
\newblock A new approach for monitoring process variance.
\newblock \emph{Journal of Statistical Computation and Simulation}, pages
  1--17, 2015.

\bibitem[Yang et~al.(2012)Yang, Cheng, Hung, and W~Cheng]{yang2012new}
Su-Fen Yang, Tsung-Chi Cheng, Ying-Chao Hung, and Smiley W~Cheng.
\newblock A new chart for monitoring service process mean.
\newblock \emph{Quality and Reliability Engineering International}, 28\penalty0
  (4):\penalty0 377--386, 2012.

\end{thebibliography}
\end{document}